\documentclass{amsart}
\usepackage{color}
\usepackage{hyperref}
\usepackage{tikz, tikz-cd}
\usetikzlibrary{arrows}
\usepackage{enumitem}
\usepackage{graphicx}

\usepackage{amsthm}
\usepackage{amssymb,amsfonts,mathrsfs,amsxtra,dsfont}
\usepackage{amsmath, amscd}

\usepackage{cleveref}

\newcommand{\fat}[1]{\mathds{#1}}

\newcommand{\RR}{\fat{R}}
\newcommand{\RRplus}{\RR_{_{\geq 0}}}

\newcommand{\antip}[1]{\mathcal{A}^{#1}}

\newcommand{\exposed}[1]{{\partial #1}}

\theoremstyle{plain}
\newtheorem{theorem}{Theorem}
\newtheorem{proposition}[theorem]{Proposition}
\newtheorem{lemma}[theorem]{Lemma}
\newtheorem{corollary}[theorem]{Corollary}

\theoremstyle{definition}
\newtheorem{definition}[theorem]{Definition}

\author{Jared Culbertson}
\address{Sensors Directorate, Air Force Research Laboratory, 2241 Avionics Circle, Bldg. 620,
Wright--Patterson Air Force Base, Ohio 45433-7302, USA.}
\email{jared.culbertson@us.af.mil}

\author{Dan P. Guralnik}
\address{Electrical \& Systems Engineering Dept., University of Pennsylvania, 200 S. 33rd St., 203 Moore Bldg.
Philadelphia, Pennsylvania 19104-6314, USA.}
\email{guraldan@seas.upenn.edu}

\author{Peter F. Stiller}
\address{Department of Mathematics, MS3368, Texas A\&M University,
College Station, Texas 77843-3368, USA.}
\email{stiller@math.tamu.edu}

\begin{document}

\title{Edge erasures and chordal graphs}

\begin{abstract}
We prove several results about chordal graphs and weighted chordal graphs by focusing on exposed edges. These are edges that are properly contained in a single maximal complete subgraph.  This leads to a characterization of chordal graphs via deletions of a sequence of exposed edges from a complete graph.   Most interesting is that in this context the connected components of the edge-induced subgraph of exposed edges are $2$-edge connected.  We use this latter fact in the weighted case to give a modified version of Kruskal's second algorithm for finding a minimum spanning tree in a weighted chordal graph.  This modified algorithm benefits from being local in an important sense.
\end{abstract}
\bigskip

\maketitle



\begin{figure}[t]
\begin{center}
	\includegraphics[width=\textwidth]{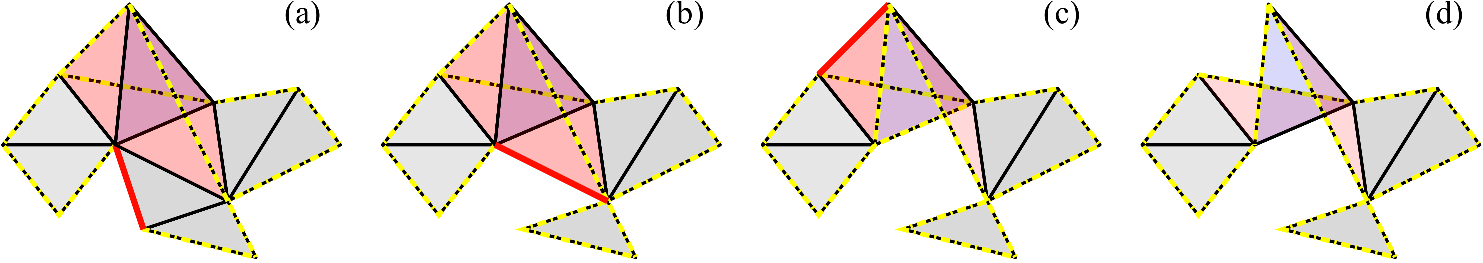}
\end{center}
    \caption{\scriptsize A sequence of three erasures of exposed edges, left to right, performed on a chordal graph (cliques illustrated as simplices for emphasis). Exposed edges being deleted are marked in red; the other exposed edges are marked with yellow dashes. Transitions (a$\to$b$\to$c) introduce new exposed edges; transition (c$\to$d) turns an exposed edge into an unexposed one. This effect is even more pronounced with larger cliques. Finally, note how all transitions are marked with stark changes to the topology of the edge-induced subgraph of exposed edges.   \normalsize\label{fig:erasures}}
\end{figure}

\section{Introduction}

In this short paper we prove several results about chordal graphs by focusing on edges which are each properly contained in a unique maximal complete subgraph; these we call exposed edges.
Our first result gives a characterization of chordal graphs as those that can be produced through a sequence of exposed edge deletions starting from a complete graph.
This characterization does not follow immediately from the usual vertex-centric characterizations of chordal graphs in terms of elimination orderings or minimal separators, and is also distinct from the important representation of chordal graphs as intersection graphs of a family of subtrees of a tree (see \cite{decaria:thesis} for a survey of these).
The edge deletions that we consider are different from the edge-without-vertex elimination orderings of related graph classes (for example, see the characterization of strongly orderable graphs in \cite{dragan:strongly_orderable}).

At first glance, one might think that exposed edges could be added to or removed from a chordal graph {\em en masse} while maintaining chordality.
%
However, the edge-induced subgraph of exposed edges in a chordal graph {\em can change dramatically through a single deletion of an exposed edge}, turning exposed edges into unexposed ones, and vice-versa (see \Cref{fig:erasures}).
Moreover,
in the reverse direction, care must be taken when adding an edge to a chordal graph to ensure that the edge is both exposed in the larger graph and that the graph remains chordal. 

Despite the unruly behavior of the edge-induced subgraph of exposed edges, we are able to circumvent the problems this creates by proving that its connected components are always $2$-edge connected. This, in turn, leads to a third result, namely a variation of Kruskal's second algorithm~\cite{kruskal:mst} for finding a minimum spanning tree in a weighted graph, with its attached relationship to ultrametrics and single-linkage clustering. This relationship is discussed in more detail in \Cref{sec:d_erasures}.

Our early investigations were motivated by theoretical work on data clustering (see \cite{CGP-projection}) and a search for an adequate notion of a minimum spanning complex for our $\antip{n}$ clustering methods, analogous to the role played by minimum spanning trees for single-linkage clustering. The results in this paper directly apply to the topological study of flag complexes obtainable by collapses from a simplex.  In fact, the language is interchangeable since the flag condition means the abstract simplicial complex is completely determined by its $1$-skeleton.  We have chosen the graph theoretical language for a more consistent presentation, but all the results can be restated topologically in terms of chordal complexes, which are flag complexes whose $1$-skeleton is a chordal graph. For example, our process of deleting an exposed edge produces a simple strong deformation retraction of the associated chordal complex. For more on this topological perspective and the relationship to simplicial collapses, see the final section of the paper. 

\section{Exposed edges in chordal graphs: erasures and edge connectivity}

We begin by collecting some basic definitions, notation, and terminology.

\begin{definition}
Let $G = (V, E)$ be an undirected simple graph (with no loops or multiple edges) having finite vertex set $V$ and edge set $E$. The degree of a vertex $v$ will be denoted by $\deg_G(v)$. The open $G$-neighborhood of a vertex $v \in V$ is 
\[
	N_G(v) = \{w \in V\setminus\{v\} \mid vw \in E\}.
\]
The closed neighborhood $N_G[v] = N_G(v) \cup \{v\}$. We will denote the induced subgraph on $A \subseteq V$ by $G[A]$. On occasion, we will simplify notation by understanding $N_G(v)$ or $N_G[v]$ to be the induced subgraph $G[N_G(v)]$ or $G[N_G[v]]$. Whether we are referring to the induced subgraph or just the vertex set will be clear from the context. In particular, complete subgraphs will occasionally be referred to as cliques. Note $G[N_G(v)]$ is sometimes called the {\em link} of $v$, particularly in a more topological setting. 

If $v_1, \ldots, v_k$ is an ordering on $V$, let $G_i = G[\{v_i, \ldots, v_k\}]$. A vertex $v$ is {\em simplicial} if the induced subgraph on $N_G[v]$ is complete. We say that a graph has a {\em perfect elimination ordering} if there is some ordering of $V$ such that $v_i$ is simplicial in $G_i$ for each $1 \leq i \leq k$. Recall also that a {\em bridge} is a cut-edge, that is, an edge whose removal increases the number of connected components of the graph. 
\end{definition}

\begin{definition}
An undirected simple graph $G$ is {\em chordal} if every induced cycle has length three. Chordality is an induced-hereditary property.  
\end{definition}

There are many characterizations of chordal graphs available in the literature. We will not attempt here to give a full survey of the relevant results, but rather point the reader to \cite{decaria:thesis}, which provides an excellent guide to the related literature. However, there is one characterization that we will need in the sequel and one implication---we combine those as a theorem here. 

\begin{theorem}[\cite{dirac:rigid_circuit,fulkerson-gross:incidence}]
\label{thm:chordal}
A graph is chordal if and only if it has a perfect elimination ordering. Moreover, any chordal graph is either complete or has two non-adjacent simplicial vertices.
\end{theorem}

Borrowing from topology, and to simplify the exposition, we refer to an edge whose endpoints induce a two-element maximal clique as a {\em facet edge}. The following lemma, however, states that for chordal graphs, the notions of bridge and facet edge are equivalent; although this is not true for an arbitrary graph (a non-bridge facet edge is in an induced cycle of length at least four). 

\begin{lemma}\label[lemma]{lem:bridge}
Let $G$ be a graph. If an edge $xy \in G$ is a bridge, then it is a facet edge. Additionally, if $G$ is chordal, then the converse holds.
\end{lemma}

\begin{definition} 
Let $G$ be a graph. An edge $xy \in G$ is said to be {\em exposed}, if $xy$ is contained in a unique maximal clique and $xy$ is {\bf not} a facet edge. We will denote the edge-induced subgraph of exposed edges of $G$ by $\exposed G$.
\end{definition}

\begin{definition} Suppose $G,H$ are graphs with the same vertex set $V$. We say that $H$ is obtained from $G$ through an {\em edge erasure}, if $G$ contains an exposed edge $e$ such that $H=G-e$.
\end{definition}

The topological nature of an erasure, which can be described in terms of a strong deformation retraction, will be discussed in \Cref{sec:topology}. We now provide a useful characterization of exposed edges. 

\begin{lemma}
\label[lemma]{lem:neighborhoods}
An edge $vw \in \exposed G$ if and only if $N_G(v) \cap N_G(w)$ is a nonempty clique in $G$. 
\end{lemma}
\begin{proof}
We remark that for any two vertices $v, w \in G$, the intersection
$N_G(v) \cap N_G(w)$ is just the union of all maximal cliques which contain both $v$ and $w$, minus $\{v,w\}$. The result follows immediately from this observation and the definitions.
\end{proof}

The previous lemma highlights that our notion of an exposed edge is weaker than that of a simplicial edge~\cite{dragan:strongly_orderable}, where the intersection is replaced by the union of the neighborhoods. Indeed, it follows from the lemma that an edge $vw$ is exposed if and only if $w$ is a non-isolated simplicial vertex of $G[N_G(v)]$, and vice versa.

\begin{theorem}
\label{thm:chordal_erasure}
A graph $H$ can be obtained from a complete graph through a sequence of erasures of exposed edges if and only if $H$ is a connected chordal graph. Throughout the erasure process each graph in the sequence remains a connected chordal graph. 
\end{theorem}
\begin{proof}
First, we can see that erasures from connected chordal graphs produce connected chordal graphs as follows. Suppose $H = G - xy$, with $G$ a connected chordal graph and $xy \in \exposed G$. If $C$ is an induced cycle in $H$ such that $\{x,y\}\nsubseteq C$ ({\em i.e.,} possibly containing $x$ or $y$, but not both), then $C$ is also an induced cycle of $G$ and so of length $3$. Otherwise, suppose $\{x, y\} \subset C$ and $|C| > 3$. Note that if $|C| > 4$, then the induced subgraph $C' = C + xy$ of $G$ has an induced cycle of length greater than $3$, a contradiction. This leaves us with the case where $C = xv_1yv_2x$ for some $v_1, v_2$. Since $xy$ is exposed in $G$, we must have $v_1v_2 \in G$, otherwise $xy$ would lie in two distinct maximal cliques and $xy$ would not be exposed. However, $v_1v_2 \in G$ (hence in $H$) means that $C$ would not be an induced cycle in $H$, a contradiction. As for connectedness, it is easy to see that an erasure does not disconnect a connected graph since by definition a bridge is not an exposed edge. 

Conversely, it suffices to show that for any non-complete connected chordal graph $G$, we can add an edge $e$ such that $e \in \exposed{(G+e)}$ with $G+e$ chordal. (Merely ensuring $e \in \exposed{(G+e)}$ does not guarantee that $G+e$ is chordal.) Given such a  $G$, suppose $v_1, \ldots, v_k$ is a perfect elimination ordering for $G$. Let $1 \leq \ell \leq k$ be the smallest index such that $G_{i}$ is complete for $i > \ell$. Then there is some $j > \ell$ with $v_{\ell}v_j \notin G$, because $G_{\ell}$ is not complete, but $G_{\ell + 1}$ is. 

Let us show that $e:={v_\ell}v_j$ is the edge we are looking for. Setting $G' = G + e$, we claim that $v_1, \ldots, v_k$ is also a perfect elimination ordering for $G'$. This will demonstrate that $G'$ is chordal, by \Cref{thm:chordal}.

For $i < \ell$, the neighbors of $v_i$ in $G'$ are just the same neighbors of $v_i$ in $G$, and $N_{G_i}[v_i]$ is a clique since $v_i$ is simplicial in $G_i$. In particular, $\{v_j, v_\ell\} \nsubseteq N_{G}(v_i)$ since $e \notin G$. Thus $v_i$ is also simplicial in $G_i'$. On the other hand, for $i > \ell$, $G_i$ (and hence $G_i'$) is complete and so every vertex is simplicial. We still need to check that $v_{\ell}$ is simplicial in $G'_{\ell}$. This follows from the fact that $v_jv_n \in G'_{\ell}$ for all $n > \ell$ since $G'_{\ell + 1}$ is complete. 

It remains to show that $e \in \exposed G'$. It is convenient to use the characterization of exposed edges given in \Cref{lem:neighborhoods}. Notice again that for $i < \ell$, we must have that $\{v_j, v_{\ell}\} \nsubseteq N_{G'}(v_i)$, since as noted above, $\{v_j, v_\ell\} \nsubseteq N_{G}(v_i)$. Hence 
\[
	N_{G'}(v_j) \cap N_{G'}(v_{\ell}) = N_{G'_{\ell}}(v_{\ell}) \setminus \{v_j\},
\]
which is a clique, as shown above, because $v_{\ell}$ is simplicial in $G'_{\ell}$.
\end{proof}

It is natural to ask whether one could retain this result while replacing the class of exposed edges with a different one. This is easily answered by noticing that the removal of a non-exposed edge either disconnects the graph (in the case of a facet edge) or results in an induced $4$-cycle.  

\Cref{thm:chordal_erasure} is similar in spirit to the result of Spinrad and Sritharan~\cite{ss:weakly_triangulated} showing that weakly chordal graphs can be recognized by the possibility of successively adding edges through the two-pair construction to arrive at a complete graph. 

\bigskip
The following observations can be derived directly from the definitions and will be useful below:
\begin{lemma}\label[lemma]{lem:crucial_observations} Let $G$ be a graph.
If $S\subset E$ is a set of facet edges, then $\exposed G = \exposed{(G-S)}$.
If $v$ is a simplicial vertex with $\deg_G(v)>1$, then every edge $e$ incident with $v$ is exposed.
\end{lemma}

\begin{proposition}
\label[proposition]{prop:tree}
Any connected chordal graph $G$ can be reduced through a sequence of erasures to a tree. 
\end{proposition}
\begin{proof} By \Cref{thm:chordal_erasure}, it suffices to verify that if $G$ is not a tree, then $G$ contains an exposed edge. Let $G'$ be the subgraph of $G$ obtained by removing all facet edges and let $G_0'$ be a connected component of $G'$ which is not a single isolated vertex. By \Cref{thm:chordal} and \Cref{lem:crucial_observations}, $G_0'$ has exposed edges, and they are also exposed in $G$.
\end{proof}

\begin{lemma}
\label[lemma]{lem:exposed}
Let $G$ be a chordal graph and $v$ a vertex in $G$ that is in a maximal clique of size at least three. Then $v$ is incident with at least two exposed edges. 
\end{lemma}
\begin{proof}
Since $v$ is in a maximal clique of size at least three, the induced subgraph $N_G(v)$ is chordal and not edgeless. Thus there is some connected component of $N_G(v)$ containing an edge, and so we can use \Cref{thm:chordal} to find two simplicial vertices $v_1, v_2$ in $N_G(v)$ in that component.  But this implies that $N_G(v) \cap N_G(v_i)$ is a (non-empty) clique for $i = 1,2$, and so $vv_1$ and $vv_2$ are exposed edges in $G$. Note that an isolated vertex $u$ in $N_G(v)$ would correspond to a facet edge $uv$ in $G$.
\end{proof}

\begin{theorem}
\label{prop:cycles}
If $G$ is a chordal graph, then every connected component of $\exposed G$ is $2$-edge connected. 
\end{theorem}
\begin{proof}
It is simple to check that the theorem holds when $G$ is either complete or has no more than four vertices.
Now suppose that $G$ is a counterexample with a minimal number of vertices (so $|G| \geq 5$).

First, we claim that $G$ has no facet edges.
By \Cref{lem:bridge}, it suffices to verify that $G$ is bridgeless. 
Suppose $e$ were a bridge (and hence not exposed in $G$).
Applying \Cref{lem:crucial_observations} with $S=\{e\}$, we have that $\exposed G = \exposed{(G-e)}$.
Now, the connected components of $G - e$ each have fewer vertices than $G$, implying that all exposed edges in $G-e$ occur in cycles. 
It follows that $G$ could not have been a counter-example---a contradiction.

Since $G$ is chordal, \Cref{thm:chordal} allows us to find non-adjacent simplicial vertices $u,v \in G$. Then $G - u$ has fewer vertices and so every exposed edge of $G-u$ is contained in a cycle of exposed edges in $G-u$. Notice that for vertices $x,y \in G-u$, we have that 
\[
	N_{G-u}(x) \cap N_{G-u}(y) = \left[N_G(x) \cap N_G(y)\right] \setminus \{u\}.
\]
and so using \Cref{lem:neighborhoods} we see that if $\{x,y\} \nsubseteq N_G[u]$, then $xy$ is exposed in $G-u$ if and only if $xy$ is exposed in $G$. Thus in this case, if $xy$ is exposed in $G$ we can find a cycle $C = xyv_1\cdots v_kx \subset \exposed{(G-u)}$. If none of the edges in $C$ are in $N_G(u)$, then $C \subset \exposed G$. However, if some edges of $C$ are contained in $N_G(u)$, then let $i$ be the smallest index with $v_{i}v_{i+1}$ in $N_G(u)$. Similarly, let $j$ be the largest index with $v_{j-1}v_{j}$ in $N_G(u)$. Notice that $j > i$, but we could have $j = i + 1$ if there is a single edge of $C$ in $N_G(u)$.  (In order to make the notation consistent, we are treating $v_0$ as $y$ and $v_{k+1}$ as $x$.) Since $G$ has no facet edges, we may apply \Cref{lem:crucial_observations} to see that the cycle $C' = xyv_1\cdots v_{i-1}v_iuv_jv_{j+1}\ldots v_kx$ is a cycle of exposed edges in $G$ containing $xy$. 

The remaining case to be considered is when $\{x,y\} \subset N_G[u]$. Here $\{x,y\} \nsubseteq N_G(v)$, since $xy \in\exposed G$ and $u$ and $v$ are not adjacent. Hence, by the same reasoning as before, we can find a cycle of exposed edges in $G$ containing $xy$ by modifying a cycle of exposed edges in $G-v$ containing $xy$.
\end{proof}

\section{Weighted chordal graphs and $w$-erasures}\label{sec:d_erasures}

We turn now to an application of these results in the setting of edge-weighted finite graphs, and show a connection with single-linkage clustering through minimum spanning trees. 

\begin{definition} Let $(G,w)$ be an edge-weighted graph, with $w\colon E\to\RRplus$, and let $H$ be a subgraph with the induced weight. We say that $H$ is obtained from $G$ through a {\em $w$-erasure}, if $H=G-e$  where $e\in \exposed G$ with $w(e) \geq w(e')$ for any exposed edge $e'$ of $G$.
\end{definition}

Observe that given any sequence $G_0,G_1, \ldots, G_m$ of graphs obtained through erasures of exposed edges $e_0, e_1,\ldots, e_{m-1}$, we can define a weighting $w$ of $G_0$ such that $G_0,G_1, \ldots, G_m$ is also a sequence of $w$-erasures. Also, recall that a minimum spanning tree for a connected weighted graph $G$ is a spanning tree which minimizes the sum of the weights over the edges of the tree. 

\begin{theorem}\label{thm:mst} Let $(G,w)$ be a weighted, connected chordal graph. If $G'$ is obtained from $G$ through a $w$-erasure, then $G'$ contains a minimum spanning tree of $(G,w)$.
\end{theorem}
\begin{proof} Let $xy \in \exposed G$ and $G' = G - xy$ be obtained by a $w$-erasure. First, recall that $G'$ is connected since any bridge in $G$ is not exposed and hence not removed in a $w$-erasure. It is also clear that the theorem holds whenever $|V(G)| \leq 3$. 

Let $T$ be a minimum spanning tree of $(G,w)$. The case of concern is when $xy \in T$. Then let $T_x,T_y$ denote the connected components of $x$ and $y$, respectively, in $T-xy$. Let $F$ denote the set of all edges $uv\in G$  with $u\in T_x$ and $v\in T_y$, excluding the edge $xy$. Since $G'$ is connected, $F$ intersects $G'$. For any $uv\in F$, the graph $T':=T-xy+uv$ is a spanning tree of $G$, implying $w_{uv}\geq w_{xy}$, by minimality of $T$. 

On the other hand, if $uv \in \exposed G$, then $w_{uv} \leq w_{xy}$ and so any exposed edge in $F$ has equal weight with $xy$. Now we can appeal to \Cref{prop:cycles} to see that $xy$ is contained in a cycle of exposed edges of $G$ which necessarily intersects $F$, say at $uv$. Thus $T' = T -xy + uv$ is another minimum spanning tree contained of $(G,w)$ in $G'$.
\end{proof}

In his seminal paper on minimum spanning trees~\cite{kruskal:mst}, Kruskal proposed two algorithms for computing such a tree. The second of which proceeds as follows: starting with the complete graph $G_0=K_n$ endowed with the weight $w$, for each $i\geq 0$ remove from $G_i$ a heaviest edge (that is, one whose $w$-value is maximal) among those not separating the current graph to obtain $G_{i+1}$. The process terminates after stage $t=\binom{n-1}{2}$ with $G_{t+1}$ a tree. Using the cut property of minimum spanning trees, it is easy to argue that every minimum spanning tree of $(K_n, w)$ may be obtained in this way. The preceding theorem then allows us to show that, surprisingly, when restricting this algorithm to only exposed edges, we are nonetheless able to recover all minimum spanning trees. 

\begin{figure}[t]
\begin{center}
	\includegraphics[width=\textwidth]{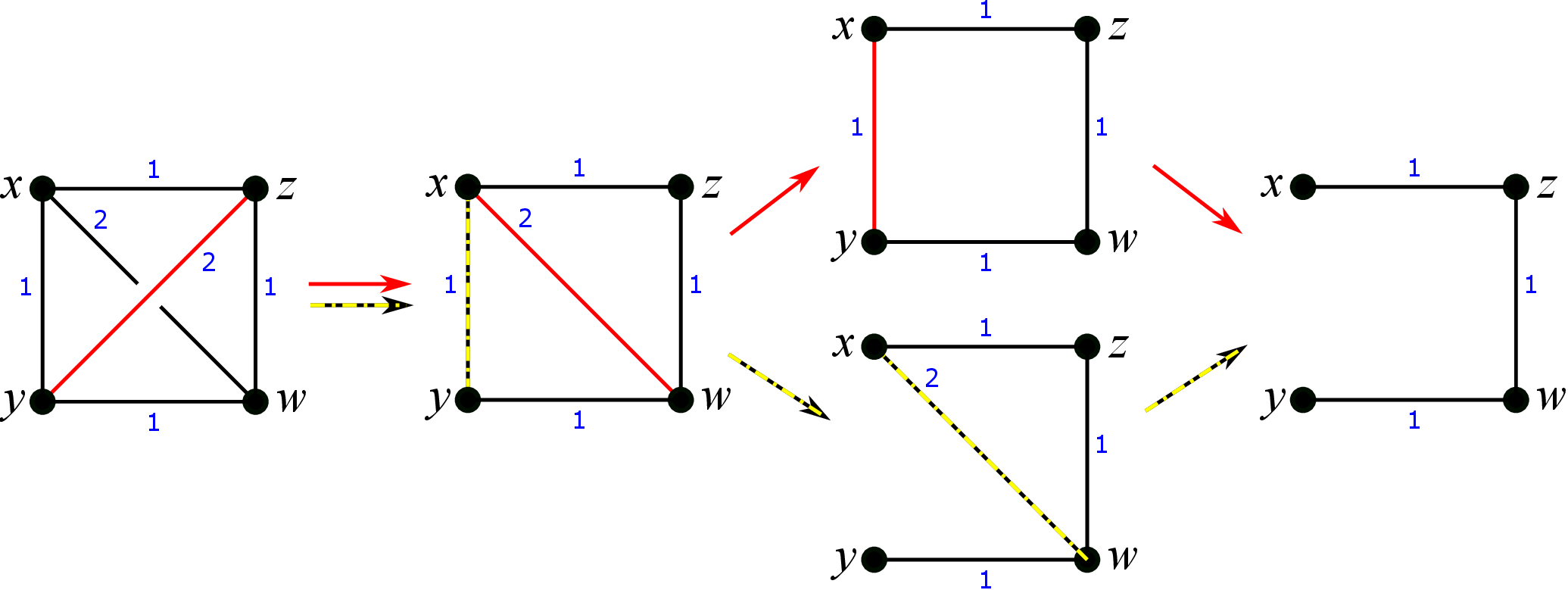}
\end{center}
    \caption{Example of a heaviest edge that is not exposed: the ``$\ell_1$ square'' drawn on the left is a weighted $K_4$ with two edges of weight $2$, both of them exposed (initially all edges are). Performing a single $w$-erasure (bottom, yellow dashed edges) results in one diameter becoming unexposed, with all the remaining edges being shorter. Proceeding with Kruskal's original algorithm (top, red edges) results in a four cycle before a minimum spanning tree is obtained, while $w$-erasures maintain chordality throughout.\label{fig:unexposed_diam}}
\end{figure}
\noindent
In so doing, we have replaced a global eligibility criterion, namely checking non-separation for a heaviest edge, with a local condition: checking whether a heaviest edge satisfies \Cref{lem:neighborhoods}. Some differences between the two algorithms are illustrated in Figure~\ref{fig:unexposed_diam}.

\begin{corollary}
Consider a weighted complete graph $(K_n,w)$, for example, the weighted graph associated with a finite metric space. Then a maximal sequence of $w$-erasures produces a minimum spanning tree for $(K_n, w)$. Any minimum spanning tree for $w$ can be obtained in this way. 
\end{corollary}
\begin{proof}
By induction, the first statement is a direct consequence of \Cref{prop:tree} and \Cref{thm:mst}. For the second, let us start with a given minimal spanning tree $T$ for $w$, and a sequence $G_0, \ldots, G_{k}$ of graphs obtained by erasure, with $G_0=K_n$ and $G_{i}$ containing $T$ for each $0 \leq i \leq k$. If $G_k \neq T$, then for any exposed edge $xy$ in $T$ of maximal weight (among the exposed edges of $G_k$), the same reasoning as in the proof of \Cref{thm:mst} (and using the same notation), shows that there must be another exposed edge $uv$ in $G_k$ with equal weight as $xy$ and $u \in T_x$, $v \in T_y$ (in particular, $uv \notin T$). Thus we can extend the sequence by setting $G_{k+1} = G_k - uv$.
\end{proof}

\section{Connections to the topological viewpoint}
\label{sec:topology}
Topologically, we can view the characterization of chordality given in \Cref{thm:chordal} in terms of perfect elimination orderings as providing  the basis for realizing chordal graphs as the $1$-dimensional skeleta of simplicial flag complexes assembled through successive ``coning-off' of existing simplices; or (by reversing the perspective) of simplicial flag complexes which admit an exceedingly tame kind of strong-deformation retraction to a vertex through a sequence of ``vertex-collapses.'' Put in the language of simple homotopy theory (see, e.g.~\cite{Kozlov-combinatorial_alg_top}, Definition 6.13 and the ensuing discussion), erasing a simplicial vertex $w$ of a chordal graph $G$ is realized in the polyhedron $|K|$ of the subtended complex $K$ as the straight-line homotopy from the identity mapping of $|K|$ to the (realization of the) simplicial map $K\to \mathrm{sd}(K)$. This homotopy fixes all vertices of $K-w$ and maps $w$ to the barycenter of its opposing face in $K$, which is the face subtended by the collection of the neighbors of $w$ in $G$, see Figure~\ref{fig:collapses}(left). 

\begin{figure}[ht]
	\includegraphics[width=\textwidth]{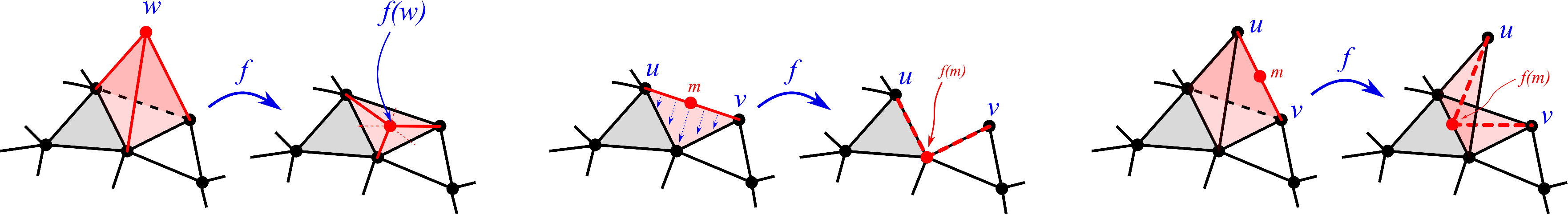}
    \caption{Collapsing an `exposed' vertex $w$ in a 3-facet (left); an exposed edge $uv$ in a 2-facet (center); and an exposed edge in a 3-facet (right). \label{fig:collapses}}
\end{figure}

Using the same language, the erasure process described in this paper can be understood as another restricted type of strong deformation retraction characterized, at the level of one-dimensional skeleta, by the removal of exposed edges. Indeed, at the level of the complexes, it quickly becomes evident that erasing an arbitrary edge of $G_i$ to obtain $G_{i+1}$ (as required by Kruskal's algorithm) does not guarantee a strong deformation retraction of $K_i=K_{G_i}$ onto $K_{i+1}$, unless the edge being removed is exposed. Here $K_i$ is the flag complex with $1$-skeleton $G_i$ and an exposed edge is one that is properly contained in a unique maximal simplex of $K_i$. Then it is possible to eliminate the edge by ``pressing in'' in the form of an {\em edge-collapse}, see Definition 6.13 in~\cite{Kozlov-combinatorial_alg_top}  and Figure~\ref{fig:collapses}(center,right). Homotopy equivalences of this kind have been studied by combinatorial algebraic topologists since the introduction of the notions of collapsibility  and simple homotopy types by Whitehead~\cite{Whitehead-collapses,Whitehead-simple_homotopy} (also see ~\cite{Kozlov-combinatorial_alg_top}, Chapter 6, for an overview and more modern treatment). Our results, then, provide an understanding of chordal graphs as $1$-skeleta of connected flag complexes arising as strong deformation retractions of a simplex, providing an interpretation of chordality from the standpoint of extendibility.

We close by briefly noting that this approach could be generalized by considering simplicial complexes other than the simplex as starting points, or {\em ambient complexes}, for the erasure process. For example, an interesting replacement would be the standard triangulation of the $n$-cube induced by its isomorphism with the Hasse diagram of the inclusion order in a power set. The corresponding question, then, is to identify which families of complexes/graphs might be characterized as emerging from some ambient complex $S$ by excavating them out of $S$ via repeated application of a restricted family of collapses, subject to a suitable stopping condition. 

\section{Acknowledgements} The authors gratefully acknowledge the support of Air Force Office of Science Research under the LRIR 15RYCOR153, MURI FA9550-10-1-0567 and FA9550-11-10223 grants, respectively.

\bibliographystyle{plain}
\bibliography{chordal.bib}

\end{document}